\documentclass[a4paper,11pt]{article}
\usepackage{graphicx}
\usepackage{amsmath,amssymb,amsthm}

\newtheorem{theorem}{Theorem}[section]
\theoremstyle{plain}
\newtheorem{conjecture}[theorem]{Conjecture}
\newtheorem{lemma}[theorem]{Lemma}
\newtheorem{proposition}[theorem]{Proposition}
\newtheorem{corollary}[theorem]{Corollary}
\newtheorem{problem}[theorem]{Problem}
\newtheorem{observation}[theorem]{Observation}
 
\theoremstyle{definition}
\newtheorem{definition}[theorem]{Definition}
\newtheorem{example}[theorem]{Example}

\begin{document}

\author{Jonas H\"agglund, Klas Markstr\"om }
\title{Shortest cycle covers and cycle double covers with large 2-regular subgraphs}

\maketitle

\begin{abstract}
	In this paper we show that many snarks have shortest cycle covers of length $\frac{4}{3}m+c$ for a constant $c$, where $m$ is the number of edges in the graph, in 
	agreement with the conjecture that all snarks have shortest cycle covers of length $\frac{4}{3}m+o(m)$.
	
	In particular we prove that graphs with perfect matching index at most 4 have cycle covers of length $\frac{4}{3}m$ and satisfy the $(1,2)$-covering conjecture of Zhang, 
	and 	that graphs with large circumference have cycle covers of length close to $\frac{4}{3}m$. We also prove some results for graphs with low oddness and discuss the 
	connection with Jaeger's Petersen colouring conjecture.
\end{abstract}

%--------------------------------------------------------------------------------------------------------------------------------------------------------
\section{Introduction}
A cycle cover of a graph $G$  is a collection of cycles $\mathcal{F}$  such that every edge edge of $G$ belongs to at least one cycle of $\mathcal{F}$.  
Cycle covers are well studied objects, see \cite{CQ,J2} for surveys, and in particular the length, i.e the sum of the lengths of the cycles in the cover, has received much attention.  We let $scc(G)$ denote the length of the shortest cycle cover of $G$.  In \cite{AT} Alon and  Tarsi made the following conjecture
\begin{conjecture}\cite{AT}
	If $G$ is a cubic 2-edge connect graph then $scc(G)\leq \frac{7}{5}m$, where m is the number of edges in $G$
\end{conjecture}\label{C:AT}
If $G$ is  3-edge-colorable it is easy to see that the conjecture is true, and in fact $scc(G)=\frac{4}{3}m$, since we can use the cycle  cover given by taking two pairs of edge colours.  This particular covering will also be a $(1,2)$-covering, ie. every edge lies in either 1 or 2 cycles, and it can  be extended to a cycle double covering, where every edge belongs to exactly 2 cycles.

In \cite{Ko} Kostochka proved that if Conjecture \ref{C:AT} is true then it also holds with the added condition that the covering must be a $(1,2)$-covering, and Zhang \cite{ZCQ} conjectured that for any 3-edge-connected cubic graph there is a shortest cycle cover which is a $(1,2)$-cover.   Conjecture \ref{C:AT} remains open, it is known from \cite{AT} that for both 2-edge and 3-edge-connected graphs the constant $\frac{7}{5}$ would be optimal, and the best current bound is $scc(G)\leq \frac{34}{21}m$, \cite{KKLNR}. 

In \cite{BGHM} all snarks (cyclically 4-edge connected, not 3 edge-colourable  cubic graphs) on $n\leq 36$  vertices were generated and, among other things, the value of $scc(G)$ was computed. It turns out that in this collection of graphs  only two have $scc(G)=\frac{4}{3}m+1$, the examples are the Petersen graph and a particular snark with 34 vertices, while all other have $scc(G)=\frac{4}{3}m$, and the authors made the following conjecture.
\begin{conjecture}\cite{BGHM}\label{C:S}
	If $G$ is a snark then 
	$scc(G)\leq \frac{4}{3}m+o(m)$
\end{conjecture}
For the class of cyclically 4-edge-connected graph this conjecture gives a significant strengthening of Conjecture \ref{C:AT}. In \cite{EM} a snark with $scc(G)=\frac{4}{3}m+2$ was constructed..

Apart from the computational results of \cite{BGHM}  we also know that Conjecture \ref{C:S} holds for snark families where the girth grows with $n$, since Jackson \cite{J1} proved that $scc(G)\leq (\frac{4}{3}+\frac{2}{g})m$ for 2-edge-connected cubic graphs of girth $g$ and $g\geq 8$. Additional conditions for having short cycle covers with different structural properties were given in \cite{St}.

In this paper we will show that Conjecture \ref{C:S} is true for several classical families of snarks  by relating short cycle covers to the perfect matching index of a cubic graph, see Section \ref{pm}, where we also show that graphs with low perfect matching index satisfy Zhang's conjecture.  We will also prove that for graphs with long cycles, of length at least $n-k$, where $k\leq 9$, we have $scc(G)\leq \frac{4}{3}m+f(k)$, and that graphs of oddness 2 have short cycle covers if they have small 2-factors.  Most of our results a proven by showing that the graphs at hand have cycle double covers which contain 2-regular subgraphs which are close to a 2-factor in size. Finally we discuss some open problems motivated by our current investigation, look at the connection to Jaeger's Petersen-colouring conjecture, and point out a perhaps unexpected rigidity in the structure of shortest covers of some snarks and graphs of low connectivity.

%--------------------------------------------------------------------------------------------------------------------------------------------------------
\section{The structure of very short covers}
Our proofs will make use of the following simple observations, neither of  which  are new and have been observed on different occasions by several other authors.

\begin{lemma}\label{lem1}
	If a cubic graph $G$ has a 2-regular subgraph $C$ and a CDC $\mathcal{C}$ such that $C\subset \mathcal{C}$ then $G$ has a $(1,2)$-cycle cover of length $2m-|C|$
\end{lemma}

Typically we will be interested in cycle covers with length close to the shortest possible value $\frac{4}{3}m$ and for the very shortest cases we can reverse the conclusion of the preceding lemma.
\begin{lemma}\label{lem2}
	If $G$ has a cycle cover $\mathcal{F}$ of length $\frac{4}{3}m+k$, where $k=0,1$ then the edges of weight 1 form a 2-regular subgraph $C$ of length $n-k$ such 
	that $C \cup \mathcal{F}$ is a cycle double cover of $G$
\end{lemma}
\begin{proof}
	Given a cycle covering $\mathcal{F}$ we define the weight of an edge $e$ to be the number of cycles from $\mathcal{F}$ which contain $e$, and the 
	weight $w(v)$ of a vertex $v$ to be the sum of  the weights of the edges incident to $v$.  The length of $\mathcal{F}$ is then given by $\frac{1}{2}\sum_v w(v)$.
	
	The lowest possible weight of a vertex is 4, which  happens if the incident edges have weights 1,1 and 2.  We now note that in a covering of these lengths all 
	edges must have weight 1 or 2, since $\frac{1}{2}\sum_v w(v)\geq \frac{4}{3}m+2$ if there exists at least one edge of weight 3.
	
	Next we note that the edges of weight 1 must form a 2-regular subgraph $C$ of $G$ with length $n-k$, and C together with $\mathcal{F}$ form a 
	cycle double cover of $G$	
\end{proof}

%--------------------------------------------------------------------------------------------------------------------------------------------------------
\section{Short covers in graphs with large circumference}

In \cite{goddynthesis} Goddyn posed what is now known as the strong cycle double cover conjecture: Given a cycle $C$ in a 2-connected cubic graph there is a cycle  double cover  of $G$ which contains $C$.   In \cite{HM,BGHM} this conjecture was verified for snarks on first 32 and then $n\leq 36$ vertices. In \cite{HM} this result was used in order to prove that if the circumference $circ(G)$ of $G$, ie the length of the longest cycle in $G$, is close enough to $n$ then $G$ has a CDC. By following the method of that proof we can also prove that graph with large circumference have short cycle covers.
\begin{theorem}
	There is a function $f(k)$ such that if the strong cycle double cover conjecture is true for cubic graphs with at most $4k$ vertices then for $G$ such that 
	$circ(G)=n-k$ we have 
	 $$scc(G)\leq \frac{4}{3}m+f(k),$$
	 and in general $f(k)\leq 4k$. 
\end{theorem}
\begin{proof}
	 Let $C$ be a longest cycle of $G$ and let $G'$ be the cubic graph which is  homeomorphic to the graph obtained by deleting all chords 
	 of $C$.  If $|C|=n-k$ then $G'$ can have at most $4k$ vertices, which happens  when the vertices not in $C$ form an independent set, 
	 and is 2-connected.  If the strong cycle double cover conjecture is true for cubic graphs  on at most $4k$ vertices then $G'$ has a CDC 
	 $\mathcal{C}_1$ which contains the cycle corresponding to $C$.
	
	The cycle $C$ together with all its chords  is homeomorphic to a hamiltonian cubic graph $G''$ with $C$ giving the hamiltonian cycle.  This graph 
	has a 3-edge-colouring with colour 1 and 2 on the edges on the hamiltonian cycle and colour 3 on the chords.  Each pair of colours define a 
	2-regular graph and those three graphs form a CDC $\mathcal{C}_2$ of $G''$ which contains the cycle corresponding to $C$.  
	
	Now $(\mathcal{C}_1\cup \mathcal{C}_2)\setminus C$ forms a cycle double cover of $G$ which contains the 2-regular graph given by the 
	colours 1 and 3 in $G''$.  This graph has size at least $n-4k$ and the theorem follows by Lemma \ref{lem1}
\end{proof}

\begin{corollary}
	 $f(k)$ exists for $k\leq 9$, and   $f(0)=0$, $f(1)=1$  
\end{corollary}
\begin{proof}
	$f[0)=0$ since in this case the graph is hamiltonian.  By Theorem B of \cite{FH} a cyclically 4-edge-connected cubic graph with an $n-1$-cycle $C$ has a CDC which contains $C$, so $f(1)\leq 1$ by Lemma  \ref{lem1}, and the Petersen graph shows that $f(1)\geq 1$.  In \cite{BGHM} it was verified that  the strong cycle double cover conjecture holds for $n\leq 36$, so $f(k)$ exists for $k\leq 9$.
\end{proof}

%--------------------------------------------------------------------------------------------------------------------------------------------------------
\section{Short covers in graphs with Oddness 2}
Recall that the \emph{oddness} $o(G)$ of a cubic graph $G$ is the minimum number of odd cycles in any 2-factor of $G$.  Graphs with oddness 0 are 3-edge-colourable and graphs  with oddness 2 and 4 are known to have cycle double covers \cite{HK,HMcG,H} and also behave well with respect to other cycle cover/decomposition properties \cite{M1}.

For some graphs of oddness 2 we can show an optimal bound on the length of the shortest cycle cover. 
\begin{theorem}
	If $G$ is cyclically 4-edge-connected and has a 2-factor with exactly two components then $scc(g) \leq \frac{4}{3}m+2$.  
	
	If there are three consecutive vertices on one the two cycles in the 2-factor with neighbours on the other cycle then $scc(g) \leq \frac{4}{3}m+1$.
\end{theorem}
There are graphs of this type with $scc(G)=\frac{4}{3}m+1$, e.g the Petersen graph, so the bound in the theorem is at most 1 away from being sharp.  The second case of the theorem will for example apply if one of the two cycles is an induced cycle in $G$, as in the Petersen graph. 
\begin{proof}
	 Let ${C_1,C_2}$ be a 2-factor with exactly two components and let $e_i,i=1,2,3$ be three edges with endpoints on both $C_1$ and $C_2$.  Let 
	 $G'$ be the cubic  graph which is homeomorphic to the graph given by $E(G)\setminus \{e_1,e_2,e_3\}$.
	 
	 Now $G'$ has a 2-factor consisting of two even cycles $D_1$ and $D_2$, corresponding to $C_1$ and $C_2$, and so has a 3-edge-colouring 
	 with colours 1 and 2 on the edges of $D_1$ and $D_2$, and colour 3 on the perfect matching $M$ given by $E(G)\setminus E(D_1\cup D_2)$.  We 
	 may assume that at least two of the three edges on which the endpoints of $e_1,e_2,e_3$ were in $C_1$ have colour 1, if necessary by switching 
	 the colour 1 and 2 on $D_1$, and likewise on $D_2$.
	 
	 The colouring now gives a cycle double covering  of $H'$, by the cycles given by each pair of the three colours, and a corresponding cycle double covering 
	 $\mathcal{C}_1$ of  
	 $E(G)\setminus \{e_1,e_2,e_3\}$ such that $C_1\subset \mathcal{C}_1$ and $C_2 \subset \mathcal{C}'_1$.  Similarly $C_1,C_2$ together with 
	 $e_1,e_2,e_3$ gives a graph with a CDC $\mathcal{C}_2$ such that $C_1\subset \mathcal{C}_2$ and $C_2 \subset \mathcal{C}_2$.  Now  
	 $\mathcal{C}_1\cup \mathcal{C}_2\setminus \{C_1,C_2\}$ is a CDC of $G$ which includes the cycles given by the colours $1,3$ in $G'$, and by our 
	 choice of colouring the length of this 2-regular subgraph is at least $n-2$, so by Lemma \ref{lem1} we have $scc(G)\leq \frac{4}{3}m+2$
	 
	 If the second case of the theorem applies then we can choose the colouring of the cycle in the condition such that all three of those edge endpoints lie on an edge of 
	 colour 1, thereby missing at most one vertex on the other cycle.  
\end{proof}

It is possible to generalize the preceding proof to the situation where the 2-factor has some even components as well, but the bound becomes weaker.
\begin{theorem}
	If $G$ is cyclically 4-edge-connected and has a 2-factor $F$ with exactly two odd components $C_1$ and $C_2$, and the total number of edges in  the three shortest disjoint paths from $C_1$ to $C_2$ in  the multigraph obtained by contracting each cycle in $F$ to a vertex is $d$, then  
	$scc(G)\leq \frac{4}{3}m+2d$.
\end{theorem}
\begin{proof}
	Let $C_1$ and $C_2$ be the two odd cycles in $F$. By Menger's theorem we can find three disjoint paths $P_1,P_2,P_3$ from $C_1$ to $C_2$ in $G$. We can now proceed as in the previous proof but using the three paths instead of the edges $e_i$.  Since some cycles in $F$ can intersect more than one of the three paths $P_i$ we cannot necessarily use the colouring modification used in the previous proof and so each edge of the three paths may contribute 2 edges which are not included in the (1,3)-coloured subgraph, thus adding $2d$ to the length of the covering.
\end{proof}
It is possible to make the bound for $h(d)$ somewhat stronger but since the method does not seem likely to give a sharp result, and quickly becomes lengthy, we have not included that analysis here.

%--------------------------------------------------------------------------------------------------------------------------------------------------------
\section{Short Covers and the Perfect Matching Index}\label{pm}

\begin{definition}
	The perfect matching index $\tau(G)$ of a bridgeless cubic graph is the smallest integer $k$ such that there exists perfect matchings $M_1,\ldots,M_k$ such that $\cup_k M_k=E(G)$
\end{definition}
Note that $G$ is 3-edge-colourable if and only if $\tau(G)=3$.  If Fulkerson's conjecture \cite{Fu} is true then $\tau(G)\leq 5$ for all bridgeless cubic graphs $G$. The perfect matching index was studied in more detail in \cite{FV}, where its value for several families of snarks were determined, and in \cite{BGHM} it was shown that all but two of the snarks with $n\leq 36$ vertices have $\tau(G)=4$.

\begin{proposition} \label{simpleprop}
	Let $M_1,M_2$ be two perfect matchings in a cubic graph $G$. Then $M_1\Delta M_2$ induces a subgraph $H$ of $G$ which  consists of disjoint even cycles. 
\end{proposition}	
\begin{proof}
	Since every vertex of $H$ either has degree $0$ or it has degree 2 and one edge belongs to $M_1$ and the other to $M_2$ the result follows. 
\end{proof}

\begin{theorem}
	Let $G$ be a cubic graph. Then $G$ has a 5-CDC where one colour class is a 2-factor if and only if $\tau(G)\leq  4$.
\end{theorem}\label{tauthm}
\begin{proof}
	If $\tau(G)=3$ then $G$ is 3-edge-colourable and the CDC given by the three pairs of edge colour classes is a 3-CDC in which each colour pair define a 2-factor.
	
	If $\tau(G)=4$ then we let $\mathcal M = \{M_1,M_2,M_3,M_4\}$ be a set of four perfect matchings such that $\bigcup_{i=1}^4 M_i=E(G)$. By the pigeonhole principle, each vertex is incident to 
	exactly one edge which is covered twice by the perfect matchings in $\mathcal M$. Therefore the set of edges covered by exactly two perfect matchings is a perfect matching which we 
	denote by $M$. Now consider $\mathcal A = \{M\Delta M_1, M\Delta M_2, M\Delta M_3, M\Delta M_4\}$. By Proposition \ref{simpleprop} $\mathcal A$ is a set of even subgraphs. If $e\in M$ 
	then $e$ is covered twice by $\mathcal A$ and if $e\not\in M$ then it is covered exactly once by $\mathcal A$. Hence $\mathcal A \cup \{E(G)\setminus M\}$ is a 5-CDC of $G$ in which 
	one colour class is a 2-factor. 
		
	Conversely assume that $G$ has a 5-CDC $\mathcal C = \{C_1,C_2, C_3, C_4, F\}$ where $F$ is a 2-factor of $G$. Then $M_i = (F \cap C_i) \cup ((E(G)\setminus F) \setminus C_i)$ is a 
	perfect matching for $i = 1, \dots, 4$. Let $\mathcal M = \{M_1,M_2,M_3,M_4\}$. If $e\in F$ then e is covered by exactly one perfect matching in $\mathcal M$ and if $e\in E(G)\setminus F$ it is 
	covered by exactly two elements of $\mathcal M$. Hence $\mathcal M$ is a perfect matching cover of $G$ with four perfect matchings. 
\end{proof}

In combination with our earlier results this shows that graphs with $\tau(G)\leq 4$ has cycle covers of the shortest possible length.
\begin{theorem}\label{th:tau}
	If $\tau(G)\leq 4$ then $scc(g)=\frac{4}{3}m$
\end{theorem}
\begin{proof}
	By Theorem \ref{tauthm} $G$ has a CDC  which contains a 2-factor and then the theorem follows by Lemma \ref{lem1}.
\end{proof}
Theorem \ref{th:tau} was implicitly present already in the PhD Thesis of Celmins \cite{Cel}, with a different proof and terminology, but as far as we have been able to determine he  did not publish it. That a low perfect matching index implies the existence of a 5-CDC, and thereby a short cycle cover, was recently also shown by Steffen in  \cite{St}.

By this theorem the snarks in several of the classical snark families have short cycle covers. We refer to e.g \cite{FV} for the full definition of these snarks families.
\begin{corollary}
	The Flower snarks, Goldberg Snarks, and Permutation snarks which are not the Petersen graph have $scc(G)=\frac{4}{3}m$
\end{corollary}
\begin{proof}
	In \cite{FV}  it was proven the snarks in each of these families have $\tau(G)=4$.
\end{proof}

As a corollary we also get a partial result on the  conjecture of Zhang \cite{ZCQ} mentioned in the introduction, as already pointed out by Steffen \cite{St}.
\begin{corollary}
	If $\tau(G)\leq 4$ then the shortest cycle cover of $G$ is a $(1,2)$-cover.
\end{corollary}

%-------------------------------------------------------------------------------------
\subsection{Cubic Graphs with lower cyclic connectivity}
Finally we note that the situation for cubic graphs of cyclic edge-connectivity exactly 2, or 3, is not as simple as one might hope. As before we only need to consider graphs which cannot be three-edge-coloured, so the only difference from the snark case is the cyclic connectivity.

There are standard reductions for cutting a cubic graph into two smaller graphs $G_2$ and $G_2$ at a 2-edge or 3-edge cut, adding in extra edges and/or a vertices to make the component graphs cubic.  Ideally one could hope that a shortest covering of $G$ could be built from shortest coverings of $G_1$ and $G_2$. However this is not always the case. In order for this composition construction to work one would need to know that for any edge $e$ in $G_1$ or $G_2$ there is a shortest cover where $e$ has a specified weight, in order to make sure that the cycle covers agree when we compose the two smaller graphs to form $G$. This is true for 3-edge-colourable cubic graphs, and the Petersen graph, but it is not true for one of the two snarks on 18 vertices which has an edge with only one possible weight in all shortest coverings.  
\begin{observation}
	In every 2-factor of the snark $S$, shown in the left part of Figure \ref{fig:sn18}, which can be extended to a cycle double covering there is a cycle which includes the edge $(a,b)$.  Hence the edge $(a,b)$ has weight 1 in every shortest cycle cover of $S$.
\end{observation}
The observation can be proven by either a case analysis or a direct computer search.  

\begin{figure}[t!]\label{fig:sn18}
	\includegraphics[scale=0.7]{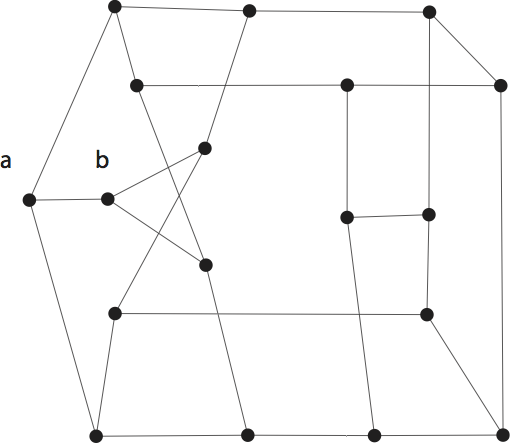}
	\includegraphics[scale=0.3]{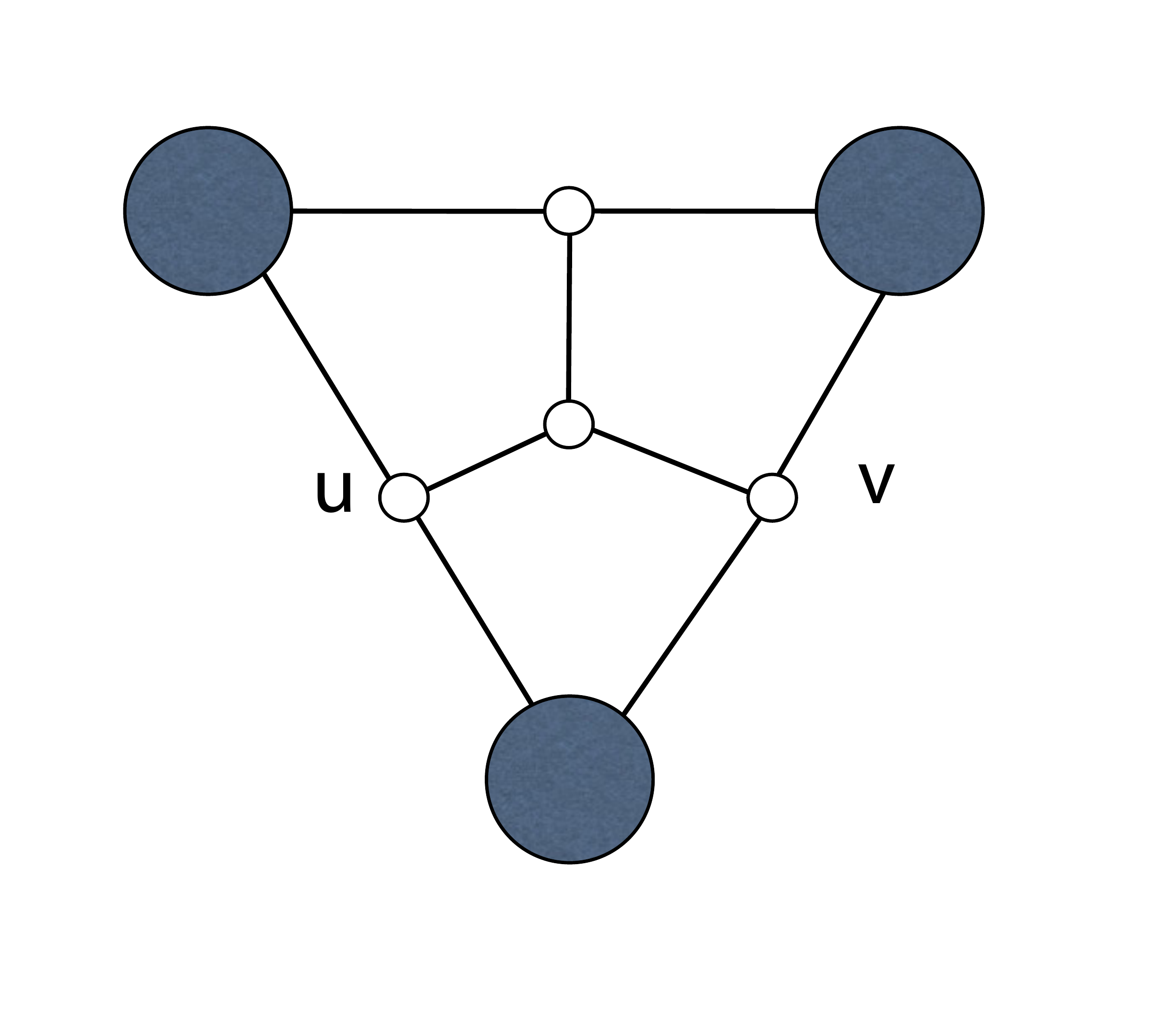}
	\caption{}
\end{figure}

Using this observation we can easily construct graphs where each component coming from a 2-edge cut has shortest cycle cover of length $\frac{4m}{3}$, while the original graph does not. 
\begin{example}\label{ex1}
	Delete the edge $(a,b)$ from three copies of the snark $S$ and then join them up to four additional vertices as shown on the right in Figure \ref{fig:sn18}, where  the 
	dark circles are the copies of $S$, to form a new graph $S_3$. 
	
	If we cut away the bottom copy of $S$, and add the edge $(a,b)$ to this copy and the edge $(u,v)$,  from this graph we are left with one large cubic graph $S_2$ which has a 2-factor 
	which corresponds  to a 2-factor in each of the copies of $S$ which includes the deleted edge $(a,b)$, and so has a cycle cover of length $\frac{4}{3}|E(S_2)|$. However each such 
	cycle cover gives weight 2 to the edge $(u,v)$, which corresponds to the  2-edge cut used to separate the third copy of $S$, and so  cannot be composed with a shortest cycle cover of $S$.  
	Hence $scc(S_3)\geq \frac{4}{3}|E(S_2)|+1$
\end{example}

In  \cite{FV} it was proven that the perfect matching index of a  graph $G$  with a 2-edge cut  is always at least the maximum of the perfect matching index of the  two graphs obtained by cutting $G$ at a 2-edge cut.  From Theorem \ref{th:tau}  and Example \ref{ex1} it  follows that the perfect matching index of a graph can be strictly greater than that of components coming from a 2-edge cut.

%--------------------------------------------------------------------------------------------------------------------------------------------------------
\section{Discussion and Problems}
A we have seen snarks with $scc(G)$ larger than $\frac{4}{3}m$ are rare among the small snarks, and the classical snark families, however we expect the minimum to increase for larger snarks.
\begin{conjecture}
	There exist snarks with $n$ vertices such that $scc(G)\geq \frac{4}{3}m + \log n $, for infinitely many $n$.
\end{conjecture}

If we restrict our attention to snarks with bounded oddness it seems reasonable that the cover length should stay close to the minimum possible, as for the snarks of high circumference. 
\begin{conjecture}
	There is a constant $c$ such that $scc(G)\leq \frac{4}{3}m + c $ if $o(G)=2$.
\end{conjecture}
\begin{problem}
	Is  there are function $h$ such that $scc(G)\leq \frac{4}{3}m+h(o(G))$?
\end{problem}

In \cite{Jae:80} Jaeger made conjecture with far reaching consequences  for the structure of cubic graphs.  A \emph{Petersen}-colouring of a cubic graph $G$ is a map $c:E(G)\rightarrow E(P)$, where $P$ is the Petersen graph, such that if $e_1,e_2,e_3$ are incident edges in $G$ then $c(e_1),c(e_2),c(e_3)$ are incident edges in $P$. Jaeger conjectured that all 2-connected cubic graphs have a Petersen-colouring. This is trivially true for 3-edge-colourable graphs and  in \cite{BGHM} the conjecture was verified for all snarks on $n\leq 36$ vertices.

Given a Petersen-coloring $c$ of $G$ and a cycle cover $\mathcal{F}$ of $P$ we can construct a cycle cover  of $G$ by taking the inverse image under $c$ of each cycle in the cover. Each edge in $P$ makes a contribution to the length of $c^{-1}(\mathcal{F})$ which is equal to its weight in $\mathcal{F}$ times the number of edges mapped onto it by $c$. Using this observation, and saying that $c$ is balanced if the same number of edges in $G$ is mapped onto every edge in $P$, we have:
\begin{theorem}
	If a cubic graph $G$ has a $P$-colouring then $scc(G)\leq\frac{7}{5}m$. If the colouring is not balanced then $scc(G)\leq\frac{7}{5}m-1$ 
\end{theorem}
The reason for the latter point is that if the colouring is unbalanced we can always find a 9-cycle in $P$ to which more than the average number edges is mapped and use that as the cycle which is covered only once in the shortest cover of $P$. 
\begin{corollary}
	If a cubic graph $G$ has a $P$-colouring and $15 \nmid m$, ie $10 \nmid n$, then  $scc(G)\leq\frac{7}{5}m-1$ 
\end{corollary}

%--------------------------------------------------------------------------------------------------------------------------------------------------------
%\section*{Acknowledgments}

%--------------------------------------------------------------------------------------------------------------------------------------------------------
\newcommand{\etalchar}[1]{$^{#1}$}

\end{document}